\documentclass[9pt,twoside,reqno]{amsart}
\usepackage{amssymb}
\usepackage{amsmath}
\usepackage{amsfonts}

\newtheorem{theorem}{Theorem}
\theoremstyle{plain}

\newtheorem{definition}{Definition}
\newtheorem{example}{Example}

\newtheorem{remark}{Remark}

\numberwithin{equation}{section}
\input{tcilatex}

\begin{document}
\title[New type of coupled fixed point theorem]{A new type of coupled fixed
point theorem in partially ordered complete metric space}
\author{Isa Yildirim}
\address{Department of Mathematics, Faculty of Science, Ataturk University,
25240 Erzurum, Turkey.}
\email{isayildirim@atauni.edu.tr}

\begin{abstract}
In this paper, we introduce a new type of coupled fixed point theorem in
partially ordered complete metric space. We give an example to support of
our result.
\end{abstract}

\maketitle

\section{Introduction and Preliminaries}

Fixed point theory in recent has developed rapidly in partially ordered
metric spaces; that is, metric spaces endowed with a partial ordering. The
first result in this direction was obtained by Ran and Reurings \cite{r1}.
They presented some applications of results of matrix equations. In \cite{r2}%
, Nieto and Lopez extended the result of Ran and Reurings \cite{r3}, for non
decreasing mappings and applied their result to get a unique solution for a
first order differential equation. While Agrawal et al. \cite{r4} and
O'Regan and Petrutel \cite{r5} studied some results for generalized
contractions in ordered metric spaces. Bhaskar and Lakshmikantham \cite{r6}
obtained some coupled fixed point results for mixed monotone operators $%
F:X\times X\rightarrow X$ which satisfy a certain contractive type
condition, where $X$ is a partially ordered metric space. They established
three kinds of coupled fixed point results: 1) existenxe theorems; 2)
existence and uniqueness theorem; and 3) theorems that ensure the equality
of the coupled fixed point components. Also, they applied results to the
study of existence and uniqueness of solution for a periodic boundary value
problem. After, Berinde \cite{b} extended the coupled fixed point theorems
for mixed monotone operators $F:X\times X\rightarrow X$ obtained in Bhaskar
and Lakshmikantham \cite{r6} by significantly weakening the involved
contractive condition.

In order to state the main result in this paper, we need the following
notions.

\begin{definition}
Let $\left( X,\leq \right) $ be a partially ordered set and endow the
product space $X\times X$ with the following partial order:%
\begin{equation*}
\text{for }\left( x,y\right) ,\left( u,v\right) \in X\times X,\text{ }\left(
u,v\right) \leq \left( x,y\right) \Leftrightarrow x\geq u,y\leq v.
\end{equation*}

We say that a mapping $F:X\times X\rightarrow X$ has the mixed monotone
property if $F\left( x,y\right) $ is monotone nondecreasing in $x$ and is
monotone increasing in $y$, that is, for any $x,y\in X$,%
\begin{equation*}
x_{1},x_{2}\in X,\text{ }x_{1}\leq x_{2}\Rightarrow F\left( x_{1},y\right)
\leq F\left( x_{2},y\right)
\end{equation*}%
and%
\begin{equation*}
y_{1},y_{2}\in X,\text{ }y_{1}\leq y_{2}\Rightarrow F\left( x,y_{1}\right)
\geq F\left( x,y_{2}\right) .
\end{equation*}
\end{definition}

\begin{definition}
An element $\left( x,y\right) \in X\times X$ is called a coupled fixed point
of the mapping $F:X\times X\rightarrow X$ if%
\begin{equation*}
F\left( x,y\right) =x\text{ \ and }F\left( y,x\right) =y.
\end{equation*}
\end{definition}

\begin{theorem}
\label{cf1} \cite{r6} Let $\left( X,\leq \right) $ be a partially ordered
set and \ suppose there is a metric $d$ on $X$ such that $\left( X,d\right) $
is a complete metric space. Let $F:X\times X\rightarrow X$ be a continuous
mapping having the mixed monotone property on $X$. Assume that there exists
a constant $k\in \lbrack 0,1)$ with%
\begin{equation}
d\left( F(x,y),F(u,v)\right) \leq \frac{k}{2}\left[ d\left( x,u\right)
+d\left( y,v\right) \right] \text{ , \ }\forall x\geq u\text{, }y\leq v.
\label{id01}
\end{equation}%
if there exist $x_{0},y_{0}\in X$ such that%
\begin{equation*}
x_{0}\leq F(x_{0},y_{0})\text{ and }y_{0}\geq F(y_{0},x_{0})\text{,}
\end{equation*}%
then there exist $x,y\in X$ such that%
\begin{equation*}
x=F(x,y)\text{ and }y=F(y,x).
\end{equation*}
\end{theorem}

In \cite{r6} Bhaskar and Lakshmikantham also established some uniqueness
results for coupled fixed points, as well as existence of fixed points of $F$
($x$ is a fixed point of $F$ if\ $F\left( x,x\right) =x$).

Inspired by above works, we derive new coupled fixed point theorems for
mapping having the mixed monotone property $F:X\times X\rightarrow X$ \ in
partially ordered metric space and we give an example to support our result.

\section{Main Results}

\begin{theorem}
\label{cf0} Let $\left( X,\leq \right) $ be a partially ordered set and \
suppose there is a metric $d$ on $X$ such that $\left( X,d\right) $ is a
complete metric space. Let $F:X\times X\rightarrow X$ be a \ continuous
mapping having the mixed monotone property on $X$. Assume that $F$ satisfies
the following condition:%
\begin{equation}
d\left( F(x,y),F(u,v)\right) \leq \delta \left( x,y,u,v\right) \left[
d\left( x,u\right) +d\left( y,v\right) \right] \text{ }.  \label{id1}
\end{equation}%
where%
\begin{equation*}
\delta \left( x,y,u,v\right) =\frac{d\left( x,F(u,v)\right) +d\left(
y,F(v,u)\right) +d\left( u,F(x,y)\right) +d\left( v,F(y,x)\right) }{1+2\left[
d\left( x,F(x,y)\right) +d\left( y,F(y,x)\right) +d\left( u,F(u,v)\right)
+d\left( v,F(v,u)\right) \right] }
\end{equation*}%
for all $x,y,u,v\in X$ with $x\geq u$ and $y\leq v.$

If there exist $x_{0},y_{0}\in X$ such that%
\begin{equation*}
x_{0}\leq F(x_{0},y_{0})\text{ and }y_{0}\geq F(y_{0},x_{0})\text{,}
\end{equation*}%
then

a) $F$ has at least a coupled fixed point there exist $\left( x,y\right) \in
X$ such that%
\begin{equation*}
x=F(x,y)\text{ and }y=F(y,x).
\end{equation*}

b) if $\left( x,y\right) ,\left( u,v\right) $ are two distinct coupled fixed
points of $F$, then $d\left( x,u\right) +d\left( y,v\right) \geq \frac{1}{4}$%
.
\end{theorem}

\begin{proof}
a) Consider the two sequences $\left\{ x_{n}\right\} $ and $\left\{
y_{n}\right\} $ in $X$ such that,%
\begin{equation}
x_{n+1}=F(x_{n},y_{n})\text{ \ and \ }y_{n+1}=F(y_{n},x_{n})  \label{iy1}
\end{equation}%
for all\ $n=0,1,2,...$.

Now, we claim that $\left\{ x_{n}\right\} $ is nondecreasing and $\left\{
y_{n}\right\} $ is nonincreasing i.e., 
\begin{equation}
x_{n}\leq x_{n+1}\text{ \ and \ \ }y_{n}\geq \text{\ }y_{n+1}  \label{iy2}
\end{equation}%
for all\ $n=0,1,2,...$. From statement of theorem, we know that $%
x_{0},y_{0}\in X$ with 
\begin{equation}
x_{0}\leq F(x_{0},y_{0})\text{ and }y_{0}\geq F(y_{0},x_{0})  \label{iy3}
\end{equation}

By using the mixed monotone property of $F$, we write%
\begin{equation}
x_{1}=F(x_{0},y_{0})\text{ \ and \ }y_{1}=F(y_{0},x_{0}).  \label{i1}
\end{equation}

Therefore $x_{0}\leq x_{1}$ and $y_{0}\geq y_{1}$. That is, the inequality (%
\ref{iy2}) is true for $n=0$.

Assume $x_{n}\leq x_{n+1}$ \ and \ \ $y_{n}\geq $\ $y_{n+1}$ for some $n$.
Now we shall prove that (\ref{iy2}) is true for $n+1$.

Indeed, from (\ref{iy2}) and the mixed monotone property of $F$, we have%
\begin{equation*}
x_{n+2}=F(x_{n+1},y_{n+1})\geq F(\text{\ }y_{n},x_{n+1})\geq
F(x_{n},y_{n})=x_{n+1}
\end{equation*}%
and%
\begin{equation*}
y_{n+2}=F(y_{n+1},x_{n+1})\leq F(\text{\ }y_{n},x_{n+1})\leq
F(y_{n},x_{n})=y_{n+1}.
\end{equation*}

Hence, by induction, $x_{n}\leq x_{n+1}$ \ and \ \ $y_{n}\geq $\ $y_{n+1}$
for all $n$.

Since (\ref{id1}), $x_{n-1}\leq x_{n}$ \ and \ \ $y_{n-1}\geq $\ $y_{n}$, we
have%
\begin{eqnarray*}
&&d\left( F(x_{n},y_{n}),F(x_{n-1},y_{n-1})\right) \\
&\leq &\left( \frac{d\left( x_{n},F(x_{n-1},y_{n-1})\right) +d\left(
y_{n},F(y_{n-1},x_{n-1})\right) +d\left( x_{n-1},F(x_{n},y_{n})\right)
+d\left( y_{n-1},F(y_{n},x_{n})\right) }{1+2\left[ d\left(
x_{n},F(x_{n},y_{n})\right) +d\left( y_{n},F(y_{n},x_{n})\right) +d\left(
x_{n-1},F(x_{n-1},y_{n-1})\right) +d\left( y_{n-1},F(y_{n-1},x_{n-1})\right) %
\right] }\right) \\
&&\left[ d\left( x_{n},x_{n-1}\right) +d\left( y_{n},y_{n-1}\right) \right]
\\
&=&\left( \frac{d\left( x_{n},x_{n}\right) +d\left( y_{n},y_{n}\right)
+d\left( x_{n-1},x_{n+1}\right) +d\left( y_{n-1},y_{n+1}\right) }{1+2\left[
d\left( x_{n},x_{n+1}\right) +d\left( y_{n},y_{n+1}\right) +d\left(
x_{n-1},x_{n}\right) +d\left( y_{n-1},y_{n}\right) \right] }\right) \left[
d\left( x_{n},x_{n-1}\right) +d\left( y_{n},y_{n-1}\right) \right] \\
&=&\left( \frac{d\left( x_{n-1},x_{n+1}\right) +d\left(
y_{n-1},y_{n+1}\right) }{1+2\left[ d\left( x_{n},x_{n+1}\right) +d\left(
y_{n},y_{n+1}\right) +d\left( x_{n-1},x_{n}\right) +d\left(
y_{n-1},y_{n}\right) \right] }\right) \left[ d\left( x_{n},x_{n-1}\right)
+d\left( y_{n},y_{n-1}\right) \right] \\
&\leq &\left( \frac{d\left( x_{n-1},x_{n}\right) +d\left(
x_{n},x_{n+1}\right) +d\left( y_{n-1},y_{n}\right) +d\left(
y_{n},y_{n+1}\right) }{1+2\left[ d\left( x_{n},x_{n+1}\right) +d\left(
y_{n},y_{n+1}\right) +d\left( x_{n-1},x_{n}\right) +d\left(
y_{n-1},y_{n}\right) \right] }\right) \left[ d\left( x_{n},x_{n-1}\right)
+d\left( y_{n},y_{n-1}\right) \right]
\end{eqnarray*}

This implies%
\begin{eqnarray}
d\left( x_{n+1},x_{n}\right) &\leq &\left( \frac{d\left(
x_{n-1},x_{n}\right) +d\left( x_{n},x_{n+1}\right) +d\left(
y_{n-1},y_{n}\right) +d\left( y_{n},y_{n+1}\right) }{1+2\left[ d\left(
x_{n},x_{n+1}\right) +d\left( y_{n},y_{n+1}\right) +d\left(
x_{n-1},x_{n}\right) +d\left( y_{n-1},y_{n}\right) \right] }\right)  \notag
\\
&&\left[ d\left( x_{n},x_{n-1}\right) +d\left( y_{n},y_{n-1}\right) \right] .
\label{iy4}
\end{eqnarray}

Similarly, from (\ref{id1}), $y_{n-1}\geq $\ $y_{n}$ and $x_{n-1}\leq x_{n}$%
, we obtain%
\begin{eqnarray}
d\left( y_{n+1},y_{n}\right) &\leq &\left( \frac{d\left(
y_{n-1},y_{n}\right) +d\left( y_{n},y_{n+1}\right) +d\left(
x_{n-1},x_{n}\right) +d\left( x_{n},x_{n+1}\right) }{1+2\left[ d\left(
y_{n},y_{n+1}\right) +d\left( x_{n},x_{n+1}\right) +d\left(
y_{n-1},y_{n}\right) +d\left( x_{n-1},x_{n}\right) \right] }\right)  \notag
\\
&&\left[ d\left( y_{n},y_{n-1}\right) +d\left( x_{n},x_{n-1}\right) \right] .
\label{iy5}
\end{eqnarray}

From this inequalities (\ref{iy4}) and (\ref{iy5}), we get%
\begin{eqnarray*}
d\left( x_{n+1},x_{n}\right) +d\left( y_{n+1},y_{n}\right) &\leq &2\left( 
\frac{d\left( x_{n-1},x_{n}\right) +d\left( x_{n},x_{n+1}\right) +d\left(
y_{n-1},y_{n}\right) +d\left( y_{n},y_{n+1}\right) }{1+2\left[ d\left(
x_{n},x_{n+1}\right) +d\left( y_{n},y_{n+1}\right) +d\left(
x_{n-1},x_{n}\right) +d\left( y_{n-1},y_{n}\right) \right] }\right) \\
&&\left[ d\left( x_{n},x_{n-1}\right) +d\left( y_{n},y_{n-1}\right) \right] .
\end{eqnarray*}

Now, let%
\begin{equation*}
\beta _{n}=2\left( \frac{d\left( x_{n-1},x_{n}\right) +d\left(
x_{n},x_{n+1}\right) +d\left( y_{n-1},y_{n}\right) +d\left(
y_{n},y_{n+1}\right) }{1+2\left[ d\left( x_{n},x_{n+1}\right) +d\left(
y_{n},y_{n+1}\right) +d\left( x_{n-1},x_{n}\right) +d\left(
y_{n-1},y_{n}\right) \right] }\right) .
\end{equation*}

Then%
\begin{eqnarray}
d\left( x_{n+1},x_{n}\right) +d\left( y_{n+1},y_{n}\right) &\leq &\beta _{n} 
\left[ d\left( x_{n},x_{n-1}\right) +d\left( y_{n},y_{n-1}\right) \right]
\label{iy6} \\
&\leq &\beta _{n}\beta _{n-1}\left[ d\left( x_{n-1},x_{n-2}\right) +d\left(
y_{n-1},y_{n-2}\right) \right]  \notag \\
&&\vdots  \notag \\
&\leq &\beta _{n}\beta _{n-1}...\beta _{1}\left[ d\left( x_{1},x_{0}\right)
+d\left( y_{1},y_{0}\right) \right]  \notag
\end{eqnarray}

Observe that $\left( \beta _{n}\right) $ is nonincreasing, with positive
terms. So, $\beta _{n}\beta _{n-1}...\beta _{1}\leq \beta _{1}^{n}$ and $%
\beta _{1}^{n}\rightarrow 0$. It follows that%
\begin{equation*}
\lim_{n\rightarrow \infty }\left( \beta _{n}\beta _{n-1}...\beta _{1}\right)
=0.
\end{equation*}

Hence, this implies that%
\begin{equation*}
\lim_{n\rightarrow \infty }\left[ d\left( x_{n+1},x_{n}\right) +d\left(
y_{n+1},y_{n}\right) \right] =0.
\end{equation*}

From this limit, we have%
\begin{equation*}
\lim_{n\rightarrow \infty }d\left( x_{n+1},x_{n}\right) =\lim_{n\rightarrow
\infty }d\left( y_{n+1},y_{n}\right) =0.
\end{equation*}

We claim that $\left\{ x_{n}\right\} $ and $\left\{ y_{n}\right\} $ are a
Cauchy sequence in $X$. Let $n<m$. Then, from the triangle inequality and (%
\ref{iy4})-(\ref{iy6}), we have%
\begin{eqnarray*}
d\left( x_{n},x_{m}\right) &\leq &d\left( x_{n},x_{n+1}\right) +d\left(
x_{n+1},x_{n+2}\right) +...+d\left( x_{m-1},x_{m}\right) \\
&\leq &\frac{\beta _{1}^{n}}{2}\left[ d\left( x_{1},x_{0})\right) +d\left(
y_{1},y_{0}\right) \right] +\frac{\beta _{1}^{n+1}}{2}\left[ d\left(
x_{1},x_{0}\right) +d\left( y_{1},y_{0}\right) \right] \\
&&+...+\frac{\beta _{1}^{m-1}}{2}\left[ d\left( x_{1},x_{0}\right) +d\left(
y_{1},y_{0}\right) \right] \\
&=&\frac{\beta _{1}^{n}}{2}\text{ \ }\left( \frac{1-\beta _{1}^{m-n}}{%
1-\beta _{1}}\right) \left[ d\left( x_{1},x_{0}\right) +d\left(
y_{1},y_{0}\right) \right] \\
&\leq &\frac{\beta _{1}^{n}}{2}\text{ \ }\left( \frac{1-\beta _{1}}{1-\beta
_{1}}\right) \left[ d\left( x_{1},x_{0}\right) +d\left( y_{1},y_{0}\right) %
\right] \\
&=&\frac{\beta _{1}^{n}}{2}\left[ d\left( x_{1},x_{0}\right) +d\left(
y_{1},y_{0}\right) \right]
\end{eqnarray*}%
and%
\begin{eqnarray*}
d\left( y_{n},y_{m}\right) &\leq &d\left( y_{n},y_{n+1}\right) +d\left(
y_{n+1},y_{n+2}\right) +...+d\left( y_{m-1},y_{m}\right) \\
&\leq &\frac{\beta _{1}^{n}}{2}\left[ d\left( y_{1},y_{0}\right) +d\left(
x_{1},x_{0}\right) \right] +\frac{\beta _{1}^{n+1}}{2}\left[ d\left(
y_{1},y_{0}\right) +d\left( x_{1},x_{0}\right) \right] \\
&&+...+\frac{\beta _{1}^{m-1}}{2}\left[ d\left( y_{1},y_{0}\right) +d\left(
x_{1},x_{0}\right) \right] \\
&\leq &\frac{\beta _{1}^{n}}{2}\left[ d\left( y_{1},y_{0}\right) +d\left(
x_{1},x_{0}\right) \right]
\end{eqnarray*}

By adding these two inequalities, we obtain%
\begin{equation*}
d\left( x_{n},x_{m}\right) +d\left( y_{n},y_{m}\right) \leq \beta _{1}^{n} 
\left[ d\left( x_{1},x_{0})\right) +d\left( y_{1},y_{0}\right) \right] .
\end{equation*}

This implies that%
\begin{equation*}
\lim_{n,m\rightarrow \infty }\left[ d\left( x_{n},x_{m}\right) +d\left(
y_{n},y_{m}\right) \right] =0.
\end{equation*}

So, $\left\{ x_{n}\right\} $ and $\left\{ y_{n}\right\} $ are indeed a
Cauchy sequence in the complete metric space $X$ and hence, convergent:
there exist $x,y\in X$ such that%
\begin{equation*}
\lim_{n\rightarrow \infty }x_{n}=x\text{ \ \ \ and \ \ }\lim_{n\rightarrow
\infty }y_{n}=y\text{.\ }
\end{equation*}

Taking limit both sides in (\ref{iy1}) and using continuity of $F$, we get%
\begin{equation*}
x=\lim_{n\rightarrow \infty }x_{n}=\lim_{n\rightarrow \infty
}F(x_{n-1},y_{n-1})=F\left( \lim_{n\rightarrow \infty
}(x_{n-1},y_{n-1})\right) =F(x,y)
\end{equation*}%
and%
\begin{equation*}
y=\lim_{n\rightarrow \infty }y_{n}=\lim_{n\rightarrow \infty
}F(y_{n-1},x_{n-1})=F\left( \lim_{n\rightarrow \infty
}(y_{n-1},x_{n-1})\right) =F(y,x).
\end{equation*}

Therefore,%
\begin{equation*}
x=F(x,y)\text{ and }y=F(y,x),
\end{equation*}%
that is, $(x,y)$ is a coupled fixed point of $F$.

b) If there exist two distinct coupled fixed points $\left( x,y\right)
,\left( u,v\right) $ of $F$, then%
\begin{eqnarray*}
d\left( x,u\right) +d\left( y,v\right) &=&d\left( F(x,y),F(u,v)\right)
+d\left( F(y,x),F(v,u)\right) \\
&\leq &\left[ d\left( x,F(u,v)\right) +d\left( y,F(v,u)\right) +d\left(
u,F(x,y)\right) \right. \\
&&\left. +d\left( v,F(y,x)\right) \right] \left[ d\left( x,u\right) +d\left(
y,v\right) \right] \\
&&+\left[ d\left( y,F(v,u)\right) +d\left( x,F(u,v)\right) +d\left(
v,F(y,x)\right) \right. \\
&&\left. +d\left( u,F(x,y)\right) \right] \left[ d\left( x,u\right) +d\left(
y,v\right) \right] \\
&=&\left[ d\left( x,u\right) +d\left( y,v\right) \right] \left[ 4d\left(
x,u\right) +4d\left( y,v\right) \right] \\
&=&4\left[ d\left( x,u\right) +d\left( y,v\right) \right] ^{2}.
\end{eqnarray*}

Therefore, we obtain that $d\left( x,u\right) +d\left( y,v\right) \geq \frac{%
1}{4}$.
\end{proof}

Now, we will give the following example for such type of mappings which
satisfy (\ref{id1}).

\begin{example}
\label{cf2}Let $X=\left\{ 0,1\right\} $ and $x\leq y\Leftrightarrow x,y\in
\left\{ 0,1\right\} $ and $x\leq y$ where "$\leq $" be usual ordering then $%
\left( X,\leq \right) $ be a partially ordered set. Let $d:X\times
X\rightarrow \left[ 0,\infty \right) $ be defined by%
\begin{eqnarray*}
d\left( 0,1\right) &=&2\text{ , }d\left( 0,0\right) =d\left( 1,1\right) =0,
\\
d\left( a,b\right) &=&d\left( b,a\right) \text{ , }\forall a,b\in X.
\end{eqnarray*}%
Then $\left( X,d\right) $ is a complete metric space.

We define $F:X\times X\rightarrow X$ as%
\begin{equation*}
F(0,0)=0\text{ , }F(0,1)=0\text{ , }F(1,0)=1\text{ , }F(1,1)=1.
\end{equation*}

Then $F$ is continuous and has the mixed monotone property. It is obvious
that $\left( 0,0\right) $, $\left( 1,0\right) $, $\left( 0,1\right) $ and $%
(1,1)$ are the coupled fixed points of $F$. If we take $x=1$, $y=0$, $u=0$
and $v=0$ then we have%
\begin{equation*}
d\left( F(x,y),F(u,v)\right) =d\left( F(1,0),F(0,0)\right) =d\left(
1,0\right) =2.
\end{equation*}

Also, for same value of $x,y,u$ and $v$, we obtain%
\begin{equation*}
\frac{k}{2}\left[ d\left( x,u\right) +d\left( y,v\right) \right] =\frac{k}{2}%
\left[ d\left( 1,0\right) +d\left( 0,0\right) \right] =k
\end{equation*}

Thus the condition $d\left( F(x,y),F(u,v)\right) \leq \frac{k}{2}\left[
d\left( x,u\right) +d\left( y,v\right) \right] $ where $\forall x\geq u$, $%
y\leq v$ of Theorem \ref{cf1} is not true for any $k\in \lbrack 0,1)$. Thus
we can not use Theorem \ref{cf1} for the mapping $F$. On the other hand we
will show that the mapping $F$ satisfies the condition of Theorem \ref{cf0}.
Now for $x=1$, $y=0$, $u=0$, $v=0$ or $\ x=1,y=0,u=1,v=1$, we have following
possibilities for values of $(x,y)$ and $(u,v)$ such that $x\geq u$ and $%
y\leq v.$

Case 1: If we take $(x,y)=(u,v)=r$ where $r=(0,0)$ or $(1,1)$ or $(1,0)$ or $%
(0,1)$, then $d\left( F(x,y),F(u,v)\right) =0$. Thus, the inequality (\ref%
{id1}) holds.

Case 2: If we take $(x,y)=(0,0),(u,v)=(0,1)$ or $(x,y)=(0,0),(u,v)=(0,1)$,
then $d\left( F(x,y),F(u,v)\right) =0$. Thus, the inequality (\ref{id1})
holds.

Case 3: If we take $(x,y)=(1,0)$ and $(u,v)=(0,0)$, then%
\begin{equation*}
d\left( F(x,y),F(u,v)\right) =d\left( F(1,0),F(0,0)\right) =d(1,0)=2,
\end{equation*}%
and%
\begin{eqnarray*}
&&\frac{d\left( x,F(u,v)\right) +d\left( y,F(v,u)\right) +d\left(
u,F(x,y)\right) +d\left( v,F(y,x)\right) }{1+2\left[ d\left( x,F(x,y)\right)
+d\left( y,F(y,x)\right) +d\left( u,F(u,v)\right) +d\left( v,F(v,u)\right) %
\right] }\left[ d\left( x,u\right) +d\left( y,v\right) \right] \\
&=&\frac{d\left( 1,F(0,0)\right) +d\left( 0,F(0,0)\right) +d\left(
0,F(1,0)\right) +d\left( 0,F(0,1)\right) }{1+2\left[ d\left( 1,F(1,0)\right)
+d\left( 0,F(0,1)\right) +d\left( 0,F(0,0)\right) +d\left( 0,F(0,0)\right) %
\right] }\left[ d\left( 1,0\right) +d\left( 0,0\right) \right] \\
&=&\frac{d\left( 1,0\right) +d\left( 0,0\right) +d\left( 0,1\right) +d\left(
0,0\right) }{1+2\left[ d\left( 1,1\right) +d\left( 0,0\right) +d\left(
0,0\right) +d\left( 0,0\right) \right] }\left[ d\left( 1,0\right) +d\left(
0,0\right) \right] \\
&=&8.
\end{eqnarray*}

Or, taking $(x,y)=(1,0),(u,v)=(0,1)$, we have%
\begin{equation*}
d\left( F(x,y),F(u,v)\right) =d\left( F(1,0),F(0,1)\right) =d(1,0)=2,
\end{equation*}%
and%
\begin{eqnarray*}
&&\frac{d\left( x,F(u,v)\right) +d\left( y,F(v,u)\right) +d\left(
u,F(x,y)\right) +d\left( v,F(y,x)\right) }{1+2\left[ d\left( x,F(x,y)\right)
+d\left( y,F(y,x)\right) +d\left( u,F(u,v)\right) +d\left( v,F(v,u)\right) %
\right] }\left[ d\left( x,u\right) +d\left( y,v\right) \right] \\
&=&\frac{d\left( 1,F(0,1)\right) +d\left( 0,F(1,0)\right) +d\left(
0,F(1,0)\right) +d\left( 1,F(0,1)\right) }{1+2\left[ d\left( 1,F(1,0)\right)
+d\left( 0,F(0,1)\right) +d\left( 0,F(0,1)\right) +d\left( 1,F(1,0)\right) %
\right] }\left[ d\left( 1,0\right) +d\left( 0,1\right) \right] \\
&=&\frac{d\left( 1,0\right) +d\left( 0,1\right) +d\left( 0,1\right) +d\left(
1,0\right) }{1+2\left[ d\left( 1,1\right) +d\left( 0,0\right) +d\left(
0,0\right) +d\left( 1,1\right) \right] }\left[ d\left( 1,0\right) +d\left(
0,1\right) \right] \\
&=&32.
\end{eqnarray*}

Therefore, the inequality (\ref{id1}) holds.

Case 4: If we take $(x,y)=(1,0)$ and $(u,v)=(1,1)$, then%
\begin{equation*}
d\left( F(x,y),F(u,v)\right) =d\left( F(1,0),F(1,1)\right) =d(1,1)=0,
\end{equation*}%
that is, the inequality (\ref{id1}) holds.

Thus all the conditions of Theorem \ref{cf0} are satisfied. Also, $F$ has
four distinct coupled fixed points $\left( 0,0\right) $, $\left( 1,0\right) $%
, $\left( 0,1\right) $ and $(1,1)$\ in $X$ and $d\left( x,u\right) +d\left(
y,v\right) \geq \frac{1}{4}$ where $\left( x,y\right) ,\left( u,v\right) $
are two distinct coupled fixed points of $F$.
\end{example}

\begin{remark}
The ratio%
\begin{equation}
\frac{d\left( x,F(u,v)\right) +d\left( y,F(v,u)\right) +d\left(
u,F(x,y)\right) +d\left( v,F(y,x)\right) }{1+2\left[ d\left( x,F(x,y)\right)
+d\left( y,F(y,x)\right) +d\left( u,F(u,v)\right) +d\left( v,F(v,u)\right) %
\right] }  \label{rt1}
\end{equation}%
might be greater or less than $\frac{1}{2}$ and has not introduced an upper
bound. If $d\left( x,u\right) +d\left( y,v\right) <\frac{1}{4}$ for every $%
x,y\in X$, then we have%
\begin{eqnarray*}
&&d\left( x,F(u,v)\right) +d\left( y,F(v,u)\right) +d\left( u,F(x,y)\right)
+d\left( v,F(y,x)\right) \\
&\leq &d\left( x,u\right) +d\left( u,F(u,v)\right) +d(y,v)+d\left(
v,F(v,u)\right) +d\left( u,x\right) \\
&&+d\left( x,F(x,y)\right) +d\left( v,y\right) +d\left( y,F(y,x)\right) \\
&=&2d\left( x,u\right) +2d\left( y,v\right) +d\left( x,F(x,y)\right)
+d\left( y,F(y,x)\right) +d\left( u,F(u,v)\right) +d\left( v,F(v,u)\right) \\
&<&\frac{1}{2}+d\left( x,F(x,y)\right) +d\left( y,F(y,x)\right) +d\left(
u,F(u,v)\right) +d\left( v,F(v,u)\right) \\
&=&\frac{1}{2}\left( 1+2\left[ d\left( x,F(x,y)\right) +d\left(
y,F(y,x)\right) +d\left( u,F(u,v)\right) +d\left( v,F(v,u)\right) \right]
\right) .
\end{eqnarray*}

It means that%
\begin{equation*}
\left( \frac{d\left( x,F(u,v)\right) +d\left( y,F(v,u)\right) +d\left(
u,F(x,y)\right) +d\left( v,F(y,x)\right) }{1+2\left[ d\left( x,F(x,y)\right)
+d\left( y,F(y,x)\right) +d\left( u,F(u,v)\right) +d\left( v,F(v,u)\right) %
\right] }\right) <\frac{1}{2}.
\end{equation*}%
which is a special case of the following Theorem \ref{cf1}. Therefore, when $%
\left( X,d\right) $ is a complete metric space such that, for all $x,y\in X$%
, $d\left( x,u\right) +d\left( y,v\right) \geq \frac{1}{4}$, the above
Theorem is valuable because (\ref{rt1}) might be greater than $\frac{1}{2}.$
\end{remark}

\begin{remark}
The example \ref{cf2} does not satisfy the conditions of Theorem \ref{cf1}.
That is, we can not say $F$ has a coupled fixed point in $X$ or not. But, we
can see that $F$ has a coupled fixed point in $X$ from Theorem \ref{cf0}. In
other words the Theorem \ref{cf0} is a generalization of Theorem \ref{cf1}.
\end{remark}

\end{document}